\documentclass[12pt,twoside]{amsart} 
 \title[Mori dream spaces of Calabi-Yau type]{Mori dream spaces of Calabi-Yau type
 and the log canonicity of the Cox rings}

\author{Yujiro Kawamata and Shinnosuke  Okawa}

\address{Graduate School of Mathematical Sciences, University of Tokyo,
Komaba, Meguro, Tokyo, 153-8914, Japan \newline
\indent
Department of Mathematics, Faculty of Science, King Abdulaziz University,
P. O. Box 80257, Jeddah 21589, Saudi Arabia}
\email{kawamata@ms.u-tokyo.ac.jp}

\address{Graduate School of Mathematical Sciences, 
University of Tokyo, 3-8-1 Komaba, Meguro-ku, Tokyo 153-8914, Japan.}

\email{okawa@ms.u-tokyo.ac.jp}

\date{\today}

\DeclareMathOperator{\Spec}{Spec}
\DeclareMathOperator{\Pic}{Pic}
\DeclareMathOperator{\Cl}{Cl}
\DeclareMathOperator{\divi}{div}
\DeclareMathOperator{\OO}{\mathcal{O}}
\DeclareMathOperator{\rank}{rank}

\usepackage{helvet}

\usepackage{latexsym}
\usepackage{amsmath}
\usepackage{amssymb}
\usepackage{amsthm}
\usepackage{amscd}
\usepackage{enumerate} 
\usepackage{amssymb} 
\usepackage{mathrsfs}          
\usepackage{color}

\newtheorem{theo}{Theorem}[section]
\newtheorem{prop}[theo]{Proposition}
\newtheorem{lem}[theo]{Lemma}

\newtheorem*{claim}{Claim}

\theoremstyle{definition}

\newtheorem{defi}[theo]{Definition}

\newtheorem{rem}[theo]{Remark}
\newtheorem*{ack}{Acknowledgments} 
\newtheorem{step}{Step}

\begin{document}
\bibliographystyle{amsalpha+}
 
 \maketitle
 
\begin{abstract}
We prove that a Mori dream space over a field of characteristic zero is of Calabi-Yau type
if and only if its Cox ring has at worst log canonical singularities. 
By slightly modifying the arguments we also reprove the characterization of
the varieties of Fano type by the log terminality of the Cox rings.
\end{abstract}


\section{Introduction}\label{intro}
A normal projective variety $X$ is said to be
of Fano type if there exists a boundary $\mathbf{Q}$-divisor $\Delta$ on $X$
such that $(X,\Delta)$ has klt singularities and $-(K_X+\Delta)$ is ample.
\cite{bchm} proved that a $\mathbf{Q}$-factorial variety of Fano type is a Mori dream space.
Then \cite[Theorem 1.1]{gost} proved that a $\mathbf{Q}$-factorial variety of Fano type
can be characterized as a Mori dream space whose Cox ring has at worst
log terminal singularities. 
In \cite{gost} they also considered the Calabi-Yau version and
proved a similar result by assuming some conjecture (see
\cite[Theorem 4.13]{gost}). The conjecture has been partially verified in \cite[Theorem 3.3]{ft},
so that their statement is generalized in the corresponding cases \cite[Theorem 4.10]{gost}.

In this paper we give a conjecture-free proof to the Calabi-Yau version
of the result. Namely, we verify the following theorem.

\begin{theo}\label{characterization}
Let $X$ be a Mori dream space over a field of characteristic zero.
Then $X$ is of Calabi-Yau type if and only if the Cox ring of $X$ has at worst log canonical
singularities.
\end{theo}
A normal projective variety $X$ is said to be of Calabi-Yau type if
$(X,\Delta)$ has log canonical singularities and $K_X+\Delta$ is numerically trivial
for some boundary $\mathbf{Q}$-divisor $\Delta$.
Note that a variety of Calabi-Yau type is not necessarily a Mori dream space.

By slightly modifying the arguments, 
we can also reprove \cite[Theorem 1.1]{gost} in Theorem \ref{gost}.
During the preparation of this paper, we received a preprint from Morgan Brown \cite{b}
in which he independently gave a different proof to the `only if' directions of Theorem
\ref{characterization} and Theorem \ref{gost}.

We explain a little bit about the proof of Theorem \ref{characterization}.
For the `only if' direction, we first take a set of $r=\rank\Pic{(X)}$ ample divisors
which are linearly independent in $\Pic{(X)}_{\mathbf{Q}}$, and
take the affine space bundle 
$\pi:Y = \text{Spec}_X\text{Sym}(\bigoplus_i \mathcal{O}_X(A_i))\to X$.
From a boundary $\mathbf{Q}$-divisor $\Delta$ on $X$ which makes the pair $(X,\Delta)$
log Calabi-Yau, we construct a boundary divisor $\Delta_Y$ on $Y$ such that
the pair $(Y,\Delta_Y)$ is also log canonical and $K_Y+\Delta_Y$ is 
$\mathbf{Q}$-linearly trivial.
By contracting the zero section of $\pi$, we obtain a birational morphism
$f: Y \to Z$ such that $(Z,\Delta_Z)$ is log canonical and $K_Z+\Delta_Z$ is 
$\mathbf{Q}$-linearly trivial for some $\mathbf{Q}$-divisor $\Delta_Z$.
It turns out that there exists a small birational morphism from the spectrum of the Cox ring of X to $Z$. Thus we prove the log canonicity of the Cox ring.

In our proof of the `if' direction of
Theorem \ref{characterization}, we first derive the $\mathbf{Q}$-effectivity of the anti-canonical divisor
of $X$ from the log canonicity of the Cox ring, by using a similar 
construction as above and applying the numerical characterization of
the pseudo-effective divisors due to \cite{bdpp} (for a Mori dream space, this result is much easier to prove;
see Proposition \ref{bdpp for Mori dream spaces}).
Since $X$ is assumed to be a Mori dream space,
we can apply the anti-canonical MMP which terminates in a semi-ample model
in the same way as in \cite{gost}. 
By the standard
facts from \cite{hk}, the semi-ample model is the quotient of an open subset of the spectrum of
the Cox ring by the dual torus action of $\Pic{(X)}$. We can assume that
the action is free, so that the log canonicity descends to the quotient. Therefore
the semi-ample model is of Calabi-Yau type. As demonstrated in \cite[Proof of Theorem 1.2]{gost},
we can trace back the anti-canonical MMP to prove that $X$ itself is of Calabi-Yau type. 

Throughout the paper, we work over a field $k$ of characteristic zero.


\begin{ack}
The second author would like to thank the co-authors of the paper \cite{gost}
for the discussion during the preparation of \cite{gost}, which was quite helpful for him
to digest the contents of the paper.

The research in this paper started from a discussion of the authors with
Caucher Birkar and Yoshinori Gongyo in the conference at
Chulalongkorn University in December 2011. We would like to thank them,
especially to Caucher and his wife Puttachat Suwankiri for hospitality.

The first author was supported by Grant-in-Aid for Scientific Research (A) 22244002.
The second author was supported by Grant-in-Aid for JSPS fellows 22-849, and
by the GCOE program ``Research and Training Center for New Development in 
Mathematics''.

\end{ack}

\section{Preliminaries}
We start with the definitions of varieties of Fano type and Calabi-Yau type.

\begin{defi}[cf. {\cite[Lemma-Definition 2.6]{prokshok-mainII}}]\label{Fano pair}
Let $X$ be a projective normal variety over a field and $\Delta$ an effective $\mathbf{Q}$-divisor on $X$
such that $K_X+\Delta$ is $\mathbf{Q}$-Cartier. 
\begin{enumerate}[(i)]
\item We say that $(X,\Delta)$ is a {\em klt Fano pair} if $-(K_X+\Delta)$ is ample and $(X, \Delta)$ is klt. 
We say that $X$ is of {\em Fano type} if there exists an effective $\mathbf Q$-divisor $\Delta$ on $X$ such that $(X,\Delta)$ is a klt Fano pair. 
\item 
We say that $X$ is of \textit{Calabi--Yau type} if 
there exits an effective $\mathbf{Q}$-divisor $\Delta$ such that $K_X+\Delta \sim_{\mathbf{Q}}0$ and $(X, \Delta)$ is log canonical.  
\end{enumerate}
\end{defi}

Next we give the definition of Mori dream spaces.

\begin{defi}[cf. \cite{hk}]\label{Mori dream space} 
A normal projective variety $X$ over a field is called a \textit{Mori dream space}
if $X$ satisfies the following three conditions:
\begin{enumerate}[(i)]
\item $X$ is $\mathbf{Q}$-factorial, $\Pic{(X)}$ is finitely generated,
and $\Pic{(X)}_{\mathbf{Q}} \simeq \mathrm{N}^1{(X)}_{\mathbf{Q}},$\label{fin_pic}
\item The nef cone $\mathrm{Nef}{(X)}$ is the affine hull of finitely many semi-ample
line bundles, 
\item there exists a finite collection of small birational maps $f_i: X \dasharrow X_i$
such that each $X_i$ satisfies (i) and (ii), and that the closed movable cone $\mathrm{Mov}{(X)}$ is the union
of the $f_i^*(\mathrm{Nef}{(X_i)})$.
\end{enumerate}
\end{defi}

The following is one of the most important properties of Mori dream spaces.

\begin{prop}{ (\cite[Proposition 1.11]{hk})}\label{mmp on mds}
Let $X$ be a $\mathbf{Q}$-factorial Mori dream space. 
Then for any divisor $D$ on $X$, a $D$-MMP can be run and terminates. 
\end{prop}

The notion of multi-section rings and the Cox rings are repeatedly used in the studies of Mori dream spaces.

\begin{defi}[Multi-section rings and Cox rings]\label{def of cox ring}
Let $X$ be an integral normal scheme. For a semi-group $\Gamma$ of Weil divisors on $X$,
the $\Gamma$-graded ring
\begin{equation*}
R(X, \Gamma)=\bigoplus_{D\in\Gamma}\text{H}^0(X,\mathcal{O}_X(D))
\end{equation*}
is called the \textit{multi-section ring} of $\Gamma$.

Suppose that the divisor class group $\Cl{(X)}$ is finitely generated.
For such $X$, choose a group $\Gamma$ of Weil divisors on $X$ such that
$\Gamma_{\mathbf{Q}}\to \Cl{(X)}_{\mathbf{Q}}$ is an isomorphism.
Then the multi-section ring $R(X, \Gamma)$ is called a \textit{Cox ring} of $X$.
\end{defi}

\begin{rem}[{See \cite[Remark 2.18]{gost} for details}]
As seen above, the definition of Cox rings depends on the choice of the group $\Gamma$.
We can prove that the ambiguity does not affect the basic properties of rings,
such as finite generation, log terminality (log canonicity), etc.
In fact, there is a canonical way to define Cox rings without ambiguity (up to isomorphisms)
due to Hausen \cite{hau}.
We can check that the properties of rings mentioned above
holds for his Cox ring if and only if ours has the same properties.
Since our definition of Cox rings as multi-section rings is easier for
calculation, we adopt our definition in this paper.
\end{rem}

We need the following intersection-theoretic characterization of
effective (resp. big) divisors on a Mori dream space.

\begin{prop}\label{bdpp for Mori dream spaces}
Let $D$ be a divisor on a Mori dream space $X$. Then
$X$ is $\mathbf{Q}$-effective (resp. big) if and only if its intersection number with
any curve in any covering family is at least (strictly greater than) zero.
\end{prop}

\begin{proof}
Note that the $\mathbf{Q}$-effectivity and the pseudo-effectivity are equivalent
for divisors on a Mori dream space.
It is enough to show the following statement.
\end{proof}

\begin{claim}
Let $X$ be a Mori dream space and
$D$ a not big but $\mathbf{Q}$-effective (resp. not $\mathbf{Q}$-effective) divisor on $X$.
Then there exists a curve in a covering family $C$ on $X$
such that $C\cdot D=0$ (resp. $C\cdot D<0$).
\end{claim}

\begin{proof}
We run a $D$-MMP
$$X=X_0\dasharrow X_1\dasharrow\cdots\dasharrow X_N,$$
so that the pushforward $D_N$ of $D$ on $X_N$ is semi-ample (resp.
there exists a $D_N$-Mori fiber space from $X_N$).

Note that all the birational maps above are surjective in codimension one.
Therefore we can find an open subset $U$ of $X_N$ whose complement
has codimension at least two such that
all the birational maps above are identities on $U$.

If $D$ is not big (resp. not $\mathbf{Q}$-effective), we take $D_N$-Iitaka
fiber space (resp. $D_N$-Mori fiber space) $f:X_N\to Y$.
Since the relative dimension of $f$ is positive, 
there is a curve $C'$ in a covering family which is contained in a fiber of $f$.
Since $X_N\setminus U$ has codimension
at least two, we can choose $C'$ so that it is contained in $U$.
Now let $C$ be the strict transform of $C'$ on $X$.
We see that this $C$ has the desired properties.
\end{proof}

\begin{rem}
The intersection theoretic characterization of pseudo-effective divisors
was first proven in \cite[Theorem 0.2]{bdpp}.
On the other hand, the characterization of big divisors in
Proposition \ref{bdpp for Mori dream spaces} does not hold for
an arbitrary variety. For example, it is known that there exists a strictly nef
divisor on a smooth projective surface which is not $\mathbf{Q}$-effective.
See \cite[Example 10.6]{har70}.
\end{rem}


\section{Proof of Theorem \ref{characterization}}
\subsection{Cox rings of Calabi-Yau Mori dream spaces}
In this subsection, we prove the `only if' direction of Theorem \ref{characterization}.

\begin{theo}\label{cox ring of cy}
Let $X$ be a Mori dream space of Calabi-Yau type.
Then Cox rings of $X$ have at worst log canonical singularities.
\end{theo}

\begin{proof}
Choose ample line bundles $A_1,\dots,A_r$ on $X$ which are
linearly independent in $\Pic{(X)}_{\mathbf{R}}$, where $r=
\rank\Pic{(X)}$.
Consider the following natural morphism
\[
f:Y = \text{Spec}_X\text{Sym}(\bigoplus_i \mathcal{O}_X(A_i))\to
Z=\Spec R(X,\Gamma'),
\]
where $\Gamma'$ is the semi-group generated by the classes of $A_i$, and let $\pi:Y\to X$ to be the structure morphism.
Since $A_i$ are ample, $f$ is a birational projective morphism which contracts
the zero section of $\pi$, which we denote again by $X$ (see \cite[Proposition 3.5]{har66}).

If $(X,\Delta)$ is a log canonical pair, then it follows that
$$\left(Y, \sum_{i=1}^{r}E_i+\pi^{*}\Delta\right)$$
is also a log canonical pair, where $E_j\subset Y$ is the divisor
corresponding to the projection $\bigoplus_{i}\OO_X(A_i)\to\bigoplus_{i\not=j}\OO_X(A_i)$.
Set $\Delta_Y=\sum_{i=1}^{r}E_i+\pi^{*}\Delta$.
Then it holds that $(K_Y+\Delta_Y)|_X=K_X+\Delta\sim_{\mathbf{Q}}0$.
Since the restriction map
$\Pic{(Y)}_{\mathbf{Q}}\to\Pic{(X)}_{\mathbf{Q}}$
is bijective, this implies that $K_Y+\Delta_Y$
is $\mathbf{Q}$-linearly trivial.
Thus we see that $K_Z+\Delta_Z=f_{*}(K_Y+\Delta_Y)$ is also $\mathbf{Q}$-linearly trivial,
where $\Delta_Z=f_{*}\Delta_Y$. Hence we get the equality
$$K_Y+\Delta_Y=f^{*}(K_Z+\Delta_Z),$$
concluding that $(Z, \Delta_Z)$ is a log canonical pair.

We shall derive the log canonicity of the Cox rings of $X$.

Set $\Gamma=\Gamma'+(-\Gamma')$. Then the multi-section ring
$R(X, \Gamma)$ is a Cox ring of $X$.
Consider the natural injective ring homomorphism
$$R(X, \Gamma')\to R(X, \Gamma).$$

\begin{claim}
The corresponding morphism
$$\Spec \, {R(X, \Gamma)}\to\Spec \, {R(X, \Gamma')}$$
is birational and contracts no divisor. 
\end{claim}
\begin{proof}
Choose an ample divisor $A_0$ from the interior of the cone
spanned by $\Gamma'$.
Then for any positive integer $m>0$ and a non-zero global section
$s\in \text{H}^0(X,\OO_X(mA_0))$, the natural ring homomorphism
$$R(X, \Gamma')_s\to R(X, \Gamma)_s$$
is an isomorphism.

By \cite[Lemma 2.7]{hk}, if we take two global sections $s_1$ and $s_2$ of $mA_0$ such that
the corresponding divisors on $X$ have no common component, then
$\{s_1,s_2\} \subset R(X, \Gamma)$ is a regular sequence. 
Therefore the divisors of $s_1$ and $s_2$ have no common irreducible components, and this concludes the proof.
\end{proof}

From the claim above, we see that the pair
$(\Spec \,  R(X, \Gamma),\Delta')$
is also log canonical, where $\Delta'$ is the effective $\mathbf{Q}$-divisor
naturally defined from $\Delta_Z$. Since $\Spec \, {R(X, \Gamma)}$ is $\mathbf{Q}$-factorial
by \cite[Proposition 2.9]{hk},
we see that $\Spec \, {R(X, \Gamma)}$ itself is log canonical.
\end{proof}

\subsection{Mori dream spaces with log canonical Cox rings}

In this subsection we prove the `if' direction of Theorem \ref{characterization}.

\begin{theo}\label{log canonical implies log Calabi-Yau}
Let $X$ be a Mori dream space whose Cox rings have at worst log canonical singularities.
Then $X$ is of Calabi-Yau type.
\end{theo}
 
\begin{proof}
\begin{step}\label{non-vanishing}
We prove that $- K_X$ is $\mathbf{Q}$-effective.

Let $C$ be an arbitrary irreducible curve on $X$ which belongs to a covering
family.
We take a set of divisors $D_1,\dots,D_r$ on $X$ with the following properties.
\begin{itemize}
\item They forms a basis of $\Cl{(X)}_{\mathbf{Q}}$, so that
if we set $\Gamma=\bigoplus_{i=1}^{r}\mathbf{Z}D_i$,
$R(X,\Gamma)$ is a Cox ring of $X$.
\item The effective cone of $X$ is contained in the cone
spanned by $D_1,\dots,D_r$, so that  
\[
R(X,\Gamma) = \sum_{d_1,\dots,d_r \ge 0} \text{H}^0(X,\sum_i d_iD_i)
\]
holds.

\item $C\cdot D_i\ge 0$ holds for $i=1,\dots, r$.
\end{itemize}

We note that the choice of the $D_i$ depends on $C$.

Let $Y$ be the total space of the direct sum $\bigoplus_i \mathcal{O}_X(-D_i)$:
\[
Y = \text{Spec}_X\text{Sym}(\bigoplus_i \mathcal{O}_X(D_i))
\]
and regard $X \subset Y$ as the zero section.
Then we have $\text{H}^0(Y,\mathcal{O}_Y) \cong R(X,\Gamma)$.
Therefore we have a natural birational
morphism $f: Y \to Z = \text{Spec }R(X,\Gamma)$ which is not necessarily
proper.
 
\begin{lem}
Let $E_i$ be the divisors on $Y$ corresponding to the divisors $D_i$.
Then the exceptional locus of $f$ is the union of the $E_i$.
\end{lem}
 
\begin{proof}
If the assertion does not hold, then there exists a curve $C'$ on $Y$ which
is not contained in the union of the $E_i$ and mapped to a point by $f$.
Since the morphism $f$ is defined by the elements of $R(X, \Gamma)$,
for any $s\in R(X, \Gamma)$, 
the value of $s$ is constant on $C'$.
We note that $C'$ is not necessarily complete.

Take a positive linear combination $D=\sum d_iD_i$ which is very ample.
For a global section $0\not= s \in \text{H}^0(X,D)$, we see that
$\divi_Y(s)=\pi^{*}\divi_X(s)+\sum d_iE_i$, where $\pi:Y\to X$
is the structure morphism.
Suppose that $\pi(C')$ is not a point. 
Then we can take a section $s \in \text{H}^0(X,D)$ such that the divisor
$\divi_{Y}(s)$ does not contain $C'$ but intersect $C'$.
Hence $s$ is not constant on $C'$, a contradiction.
If $\pi(C')$ is a point, $C'$ is a curve in the fiber
$\pi^{-1}(\pi(C'))$, which is identified with $\mathbf{A}^r$ by
choosing a non-zero local section of $\OO_X(D_i)$ around $\pi(C')$ for each $i$.
If $\pi(C')$ is not contained in $\divi_X(s)$,
we see that the restriction of $s$ to the fiber is a non-zero monomial
of exponent $(d_1,\dots, d_r)$ under the identification.
Since $C'$ is assumed to be not contained in the coordinate hyperplanes,
we can find a suitable exponent $(d_1,\dots, d_r)$ such that
the monomial is not constant on $C'$.
\end{proof}
 
Since $Z$ has only log canonical singularities,
we can write $f^*K_Z = K_Y+\sum_i e_iE_i$ such that $e_i \le 1$ for all $i$.
Then $(K_Y + \sum_i e_iE_i) \vert_X \sim_{\mathbf{Q}} 0$ because
$X$ is mapped to a point by $f$.
By the adjunction formula, we have
$K_X \sim_{\mathbf{Q}} - \sum (1-e_i)D_i$.
 
It follows that $(K_X \cdot C) \le 0$.
Since $C$ was arbitrary, we conclude that $-K_X$ is $\mathbf{Q}$-effective by
Proposition \ref{bdpp for Mori dream spaces}.
\end{step}

\begin{step}\label{after non-vanishing}

Since $X$ is a Mori dream space,
we have a ($-K_X$)-MMP
$$X=X_0\dasharrow X_1\dasharrow \cdots\dasharrow X_N,$$
where each step is a birational map and $-K_{X_N}$ is semi-ample, 
because $-K_X$ is $\mathbf{Q}$-effective.

Fix a Cox ring $R(X, \Gamma)$ of $X$.
We recall some facts on the GIT of Cox rings from
\cite[Proposition 2.9]{hk}. First of all,
there exists a canonical action of the torus $T=\text{Hom}{(\Gamma,k^{*})}$
on the affine variety $\Spec \, {R(X, \Gamma)}$, and $X_N$ is the categorical quotient by $T$
of the semi-stable locus $U\subset\Spec \, {R(X, \Gamma)}$ with respect to a character of $T$
which corresponds to an ample divisor on $X_N$ (\cite[Proof of Proposition 2.9]{hk}).

Since $X$ is $\mathbf{Q}$-factorial and $X_N$ is obtained from a MMP starting from $X$,
we see that $X_N$ also is $\mathbf{Q}$-factorial. 
 $\mathbf{Q}$-factoriality of $X_N$, in turn,
implies that the quotient $U\to U/T=X_N$ is the geometric quotient (see \cite[Proposition 1.11(2)]{hk}
and \cite[Corollary 2.4]{hk}). Moreover, by replacing $\Gamma$ with its
subgroup of finite index if necessary, we can assume that the torus $T$ acts on $U$ freely
(\cite[Proposition 2.9]{hk}). Therefore the quotient morphism $U\to U/T=X_N$ is smooth.

Since $\Spec \, {R(X, \Gamma)}$ is log canonical, so is its open subset $U$.
Hence we see that $X_N=U/T$ has log canonical singularities.
By taking a general effective $\mathbf{Q}$-divisor $\Delta_N$ on $X_N$ which is
$\mathbf{Q}$-linearly equivalent to $-K_{X_N}$, we obtain a log Calabi-Yau pair $(X_N,\Delta_N)$
due to the semi-ampleness of $-K_{X_N}$.

Finally we trace back the ($-K_X$)-MMP as in \cite[Proof of Theorem 1.2]{gost}, showing that
$X$ itself is of Calabi-Yau type.
\end{step}
\end{proof}


\section{Characterization of varieties of Fano type revisited}

Using the similar arguments as above, we can reprove the
characterization of varieties of Fano type ($=$\cite[Theorem 1.1]{gost}).

\begin{theo}\label{gost}
Let $X$ be a $\mathbf{Q}$-factorial normal projective variety over $k$.
$X$ is of Fano type if and only if the Cox ring of $X$ is of finite type and
log terminal.
\end{theo}

\begin{proof}
For `only if' direction it is enough to prove that the Cox ring of $X$ is log terminal,
since the finite generation is proved in \cite[Corollary 1.3.2]{bchm}.
We only point out which part of the proof of Theorem \ref{cox ring of cy} should
be modified.

Instead of thinking of the pair
$$\left(Y, \ \sum_{i=1}^{r}E_i+\pi^{*}\Delta\right),$$
we should think of the following klt pair
\[
\left(Y, \ \Delta_Y:=(1-\epsilon)\sum_{i=1}^{r}E_i+\pi^{*}\Delta\right)
\]
for some positive number $\epsilon<1$. 

Then it holds that
$(K_Y+\Delta_Y)|_X=K_X+\Delta+\epsilon\sum_iA_i$, since we have
$\mathcal{O}_Y(E_i)|_X\simeq\mathcal{O}_X(-A_i)$.
By replacing $A_i$ from the beginning if necessary,
we can assume that there exists $\epsilon$ such that
$-(K_X+\Delta)\sim_{\mathbf{Q}}\epsilon\sum_iA_i$, so that
$(K_Y+\Delta_Y)|_X\sim_{\mathbf{Q}}0$.

Therefore we get the equality
$$K_Y+\Delta_Y=f^{*}(K_Z+\Delta_Z)$$
as before, where $\Delta_Z=f_{*}\Delta_Y$.
Hence the pair $(Z, \Delta_Z)$ is klt.
The rest of the proof is the same.

For the proof of the `if' direction, we first prove that
$-K_X$ is big.
If we carry out the same arguments as in Step \ref{non-vanishing} of the proof
of Theorem \ref{log canonical implies log Calabi-Yau},
by the log terminality of the Cox ring we see that $e_i<1$ holds for all $i$.
Since some positive linear combination of the $D_i$ is ample, we deduce that
$C\cdot (-K_X)=C\cdot (\sum (1-e_i)D_i)>0$.
Since $X$ is a Mori dream space, this implies the bigness of $-K_X$ by
Proposition \ref{bdpp for Mori dream spaces}.

Next, by arguing as in Step \ref{after non-vanishing},
we arrive at a model $X_N$ such that $-K_{X_N}$ is semi-ample and big.
We can also show that $X_N$ is log terminal from the log terminality of the Cox ring,
concluding that $X_N$ is of Fano type. The rest is the same.
\end{proof}


\end{document}